\documentclass{article}

\usepackage{PRIMEarxiv}

\usepackage[english]{babel}
\usepackage{microtype}
\usepackage{booktabs}
\usepackage{etoolbox}

\usepackage[T1]{fontenc}

\usepackage{graphicx}
\usepackage{diagbox}
\usepackage{booktabs}
\usepackage{amsmath}
\usepackage{amssymb}
\usepackage{amsthm}
\usepackage{subcaption}
\usepackage{hyperref}
\usepackage{siunitx}
\usepackage{algorithm}
\usepackage{algpseudocode}

\newcommand{\R}{\mathbb{R}}

\newcommand{\K}{\mathcal{K}}

\newcommand{\pe}{\mathrm{pe}}
\newcommand{\PE}{\mathrm{PE}}
\newcommand{\PET}{\mathrm{PET}}

 \newtheorem{theorem}{Theorem}

 \newtheorem{prop}{Proposition}
 \newtheorem{defn}{Definition}
 \newtheorem{rem}{Remark}

\usepackage{xcolor}

\fancypagestyle{firstpage}{
  \fancyhf{} % limpia todo
}

\title{Persistent Entropy Transform: An entropy-based descriptor for topological data analysis
}

\author{ \href{https://orcid.org/0009-0006-1316-9026}{\includegraphics[scale=0.06]{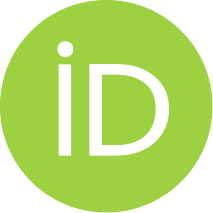}}Victor Toscano-Duran\thanks{Corresponding author.}, \href{https://orcid.org/0000-0001-9937-0033}{\includegraphics[scale=0.06]{orcid.pdf}}Rocio Gonzalez-Diaz \\
  Department of Applied Mathematics I, University of Seville \\
  Seville, Spain \\
  \texttt{\{vtoscano, rogodi\}us.es} \\
   \And
  \href{https://orcid.org/0000-0002-3624-6139}{\includegraphics[scale=0.06]{orcid.pdf}}Miguel A. Gutiérrez-Naranjo \\
  Department of Computer Science and Artificial Intelligence, University of Seville \\ Seville, Spain \\
  \texttt{magutier@us.es} \\
}

\begin{document}
\maketitle
\thispagestyle{firstpage}

\begin{abstract} 

Persistent entropy provides a compact 
%and stable information-theoretic 
summary of persistence diagrams, but 
%fundamentally 
discards 
%directional and 
geometric information inherent to the data.
%filtration process. 
This limitation creates a gap between scalar summaries, which are computationally efficient but geometrically coarse, and directional topological transforms, which are expressive but high-dimensional and computationally demanding. 
In this work, we introduce the \emph{Persistent Entropy Transform} (PET), a novel directional topological descriptor 
% Creo que si no se explican los uficiente, mejor quitar las fórmulas del abstract (M.A.)
%
%defined as the functional
%\[
%\mathrm{PET}_X : S^{d-1} \to \mathbb{R},
%\qquad
%v \mapsto PE\!\left(\mathrm{D}_k(f_v)\right),
%\]
%where $f_v(x)=\langle x,v\rangle$ is the directional height function and $PE$ denotes persistent entropy. 
which can be interpreted as an entropy-based compression of
directional topological transforms.

We establish basic theoretical properties of PET.
%under explicit finite-diagram and bounded-persistence assumptions. 
In particular, we prove translation invariance, scale invariance under positive uniform scalings, and orthogonal equivariance. 
%stability with respect to indexed perturbations of vertex positions on a fixed underlying complex. 
%The latter yields finite-direction approximation guarantees controlled by the covering radius of the sampled directions, providing a principled link between directional sampling density and discretization error.

Empirically, we use synthetic shapes to assess directional sensitivity, consistency with rotational equivariance, robustness under controlled perturbations, and
dependence on directional sampling density. In addition, we test the novel tool on
%use 
two real time-series benchmark datasets, the TwoLeadECG and the MIT-BIH, 
%in order 
to provide a proof of concept showing that PET embeddings can be used  succesfully as compact  feature vectors on real time-series benchmarks. 
These experiments support and validate PET as a compact and computationally tractable descriptor. 
%rather than to claim universal superiority over existing classification methods or fully expressive directional transforms.


\end{abstract}
\title{Persistent Entropy Transform}
\author{Victor Toscano-Duran}
\date{May 2026}

\maketitle

\section{Introduction}
Topological Data Analysis (TDA) provides mathematically rigorous tools for extracting
structural and geometric information from complex data
\cite{edelsbrunner2010computational,chazal2021introduction,
carlsson2020topological}. 
%Its central methodology, 
One of its main methodologies, persistent homology,
captures the evolution of topological features across scales and encodes them through
persistence diagrams and barcodes
\cite{edelsbrunner2002topological,ghrist2008barcodes,
edelsbrunner2008persistenthomology}. These descriptors enjoy strong theoretical
foundations, including stability under perturbations of the filtering function
\cite{cohen-steiner2007stability,Chazal2014Stability}
and structural stability results for persistence modules and geometric complexes
\cite{chazal2016structure}. As a consequence, persistent homology has become an
important component in modern machine learning, signal analysis, and geometric inference
pipelines
\cite{TopologicalDeepLearning2024,hofer2017DeepLearningTopologicalSignatures,
gabrielsson2020Topologylayerformachinelearning}.

\medskip

Despite these advantages, persistence diagrams are not directly compatible with
standard statistical and machine learning methods due to their variable cardinality and
non-Euclidean geometry. This limitation has motivated a large body of work on
vectorizations and functional summaries of persistence diagrams.
Persistence landscapes \cite{JMLR:v16:bubenik15a}
embed persistence information into Banach spaces of piecewise-linear functions,
while persistence images \cite{JMLR:v18:16-337}
provide stable finite-dimensional representations through kernelized density estimates.
Other approaches include algebraic embeddings
\cite{DBLP:conf/iciap/FabioF15},
kernel methods based on optimal transport
\cite{pmlr-v70-carriere17a},
and differentiable formulations for deep learning
\cite{10.1007/s10208-021-09522-y,
hofer2017DeepLearningTopologicalSignatures,
gabrielsson2020Topologylayerformachinelearning}.
While these methods successfully convert persistence diagrams into usable features,
they share an important limitation: they summarize topological information obtained from
a \emph{fixed filtration}. As a consequence, they only indirectly capture anisotropic
geometric structure and are inherently insensitive to directional variability.

\medskip

A complementary line of research addresses this limitation through
\emph{directional topological transforms}, which evaluate topological descriptors across
families of directional filtrations.
Let $d$ denote the ambient dimension. The Persistent Homology Transform (PHT)
\cite{Turner2014-ws} and the Euler Characteristic Transform (ECT)
\cite{ghrist2018euler,Curry2022-jl,Munch02012025}
associate a topological summary with each direction
$v\in S^{d-1}$, where $S^{d-1}$ denotes the unit sphere in $\mathbb{R}^d$, via the
height function $f_v(x)=\langle x,v\rangle$.
These transforms encode rich anisotropic geometric information and admit strong
injectivity guarantees under suitable assumptions
\cite{Turner2014-ws,Curry2022-jl}.
However, this expressiveness comes at a significant computational cost.
The output is a high-dimensional functional object indexed by $S^{d-1}$, typically
requiring substantial storage, discretization, and comparison procedures.
In practice, this limits scalability and complicates integration with standard
statistical pipelines.

\medskip

At the opposite extreme, scalar descriptors such as \emph{persistent entropy} (PE)
\cite{chintakunta2015entropybarcode}
provide extremely compact and computationally efficient summaries.
PE measures the dispersion of persistence intervals through a Shannon
entropy functional and has been successfully applied to noise discrimination, signal
analysis, and biological systems
\cite{rucco2017newtopologicalentropyformeasuringsimilarities,
rucco2016characterisationpersistententropy,
merelli2015topologicalTopologicalCharacterizationofComplexSystems,
atienza2019persistent}.
%Moreover, it enjoys rigorous stability guarantees \cite{atienza2020stabilitypersistententropy}.
Moreover, persistent entropy is continuous with respect to perturbations of
the input data \cite{atienza2020stabilitypersistententropy}. 
Yet this compactness comes at a fundamental price: PE collapses an entire
persistence diagram into a single scalar value, thereby discarding directional and
geometric information. In particular, geometrically distinct shapes may be
indistinguishable whenever their persistence-length distributions coincide.

\medskip

\noindent
\textbf{Our contribution: entropy-based compression of directional topology.}
This tension between expressiveness and compactness motivates the present work.
We identify a previously unexplored regime in topological descriptors:
\emph{directional entropy-based summaries} that retain anisotropic information while
remaining computationally and statistically tractable. In this work we introduce the \emph{Persistent Entropy Transform} (PET), defined for a
shape or signal $X \subset \mathbb{R}^d$ and a fixed homology degree $k$, with $0\leq k\leq d$ as
\[
\mathrm{PET}^{(k)}_X : S^{d-1} \rightarrow \mathbb{R},
\qquad
v \mapsto PE^{(k)}\!\left(\mathrm{D_k}(f_v)\right),
\]
where $f_v(x)=\langle x,v\rangle$ is the directional height function,
$\mathrm{D_k}(f_v)$ denotes the associated persistence diagram at homology degree $k$,
and $PE$ is the persistent entropy functional.

Conceptually, we refer to this construction simply as $\mathrm{PET}$ whenever the
%ambient dimension $d$ and homological 
homology degree $k$ and the shape $X$ are fixed or clear from context.
PET can be understood as an \emph{entropy-based compression of directional
topological transforms}.

The directional formulation of $\mathrm{PET}$ induces a structured functional
object on the sphere $S^{d-1}$, raising questions regarding rotational equivariance,
directional regularity, discretization error, and the trade-off between compression
and geometric expressiveness. Compared to scalar summaries, $\mathrm{PET}$ captures anisotropic structure and compared to PHT and ECT, it dramatically reduces representation size and computational cost, while retaining sufficient geometric information for many tasks. 

Besides, $\mathrm{PET}$ retains a single scalar per
direction, resulting, after discretization over $N$ points of $S^{d-1}$, in a vector representation in
$\mathbb{R}^N$ that is directly compatible with standard machine learning pipelines.
This makes PET particularly suitable for large-scale or time-constrained applications.
%where full directional transforms are impractical. 

Moreover, PET inherits the corresponding PE continuity property. This provides
a natural form of robustness to small perturbations of the input data.
%, without requiring or claiming a stability bound for the persistent entropy functional itself.

% Table~\ref{tab:transforms} summarizes these trade-offs.

% \begin{table}[ht]
% \centering
% \caption{Comparison of topological descriptors. $N$ denotes the number of sampled directions, $n$ the number of simplices, and $\omega$ the matrix multiplication exponent. Computational costs depend on the chosen persistence reduction algorithm and the structure of the underlying complex.}
% \label{tab:descriptor_comparison}

% \begin{tabular}{lllll}
% \toprule
% \textbf{Descriptor} & \textbf{Output} & \textbf{Stability} & \textbf{Injectivity} & \textbf{Complexity} \\
% \midrule
% PE  & Scalar $\in \mathbb{R}$ &
% Lipschitz &
% No &
% $\mathcal{O}(n\log n)$ \\

% ECT & Function $S^{d-1}\rightarrow\mathbb{Z}$ &
% Stable &
% Yes (finite directions) &
% $\mathcal{O}(Nn\log n)$ \\

% PHT & Function $S^{d-1}\rightarrow\mathrm{D}$ &
% Stable &
% Yes (under suitable assumptions) &
% $\mathcal{O}(Nn^{\omega})$ \\

% PET & Function $S^{d-1}\rightarrow\mathbb{R}$ &
% Lipschitz (inherited) &
% Open &
% $\mathcal{O}(Nn\log n)$ (practical) \\
% \bottomrule
% \end{tabular}
% \label{tab:transforms}
% \end{table}

% \medskip

Figure~\ref{fig:motivation} illustrates the motivation behind PET. As can be seen, the circle and the ellipse have the same PE value, whereas the ellipse and its rotated version have different PE values.
Nevertheless, the circle and the ellipse 
become
distinguishable through their directional entropy profiles, which encode anisotropy in
a compact form. 
%The figure shows the PET computed for a circle, an ellipse, and a rotated ellipse using directional lower-star filtrations sampled uniformly over $S^1$. As you can see, the circle and the ellipse has the same PE, but the ellipse and its rotated version doesn't. In contrast, 
Using PET, 
the circle produces an approximately constant PET, reflecting its rotational symmetry. In contrast, the ellipse exhibits a periodic PET profile associated with its anisotropic geometry. Rotating the ellipse results in a phase shift of the transform, illustrating the rotational equivariance of PET.

\begin{figure}[ht!]
    \centering
    \begin{subfigure}[b]{0.45\textwidth}
    \centering
    \includegraphics[width=\textwidth]{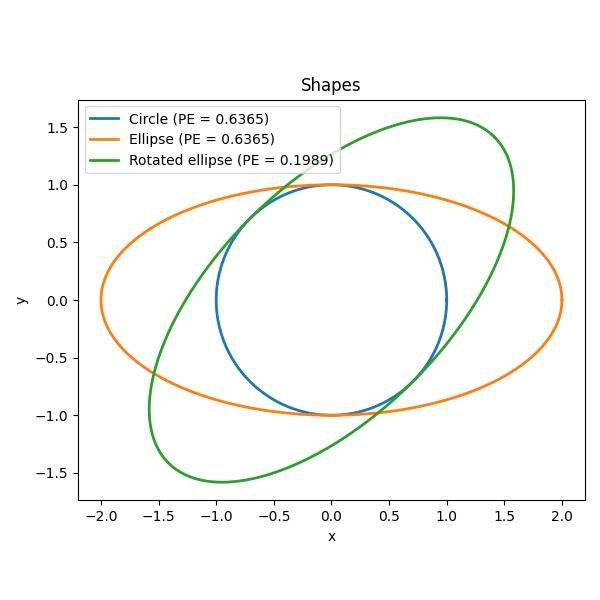}  
    \caption{}
    \label{fig:comparingPEsInitial}
    \end{subfigure}
    \begin{subfigure}[b]{0.45\textwidth}
    \centering
    \includegraphics[width=\textwidth]{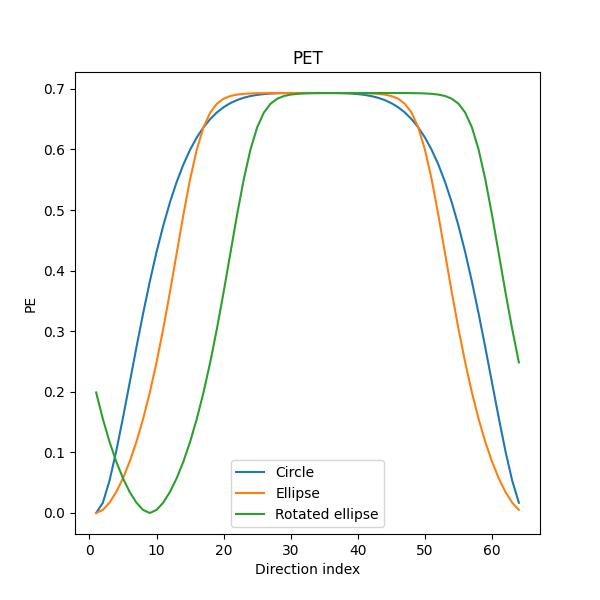}
    \caption{}
    \label{fig:comparingPEsAproxPE}
    \end{subfigure}
    \caption{(a) Three shapes: a circle, an ellipse, and a rotated ellipse together with the associated PE. (b) The PET associate to the three shapes of (a). The ellipse exhibits direction-dependent variation. The PET profile of the rotated ellipse is a reparameterized version of that of the original ellipse. The circle produces a different PET profile.} \label{fig:motivation}
\end{figure}

\paragraph{Contributions.} The main contributions of this work are as follows: 

\begin{itemize} 
\item \textbf{Persistent Entropy Transform (PET).} We introduce PET as a novel class of directional, entropy-based topological descriptors.
%, filling the gap between scalar summaries and high-dimensional directional transforms. 
\item \textbf{PET theoretical properties.} We establish translation invariance, rotational equivariance, scale invariance, and  continuity of PET with respect to small perturbations of the
input data.
%associated persistence diagrams by combining persistence diagram stability with continuity results from persistent entropy \cite{cohen-steiner2007stability,atienza2020stabilitypersistententropy}. 
\item \textbf{Directional discretization and regularity.} We analyze the approximation of PET under finite directional sampling and discuss the role of directional regularity in controlling discretization error.  
\item \textbf{Empirical validation.} We demonstrate that PET captures anisotropic geometric structure, remains robust under small perturbations, and provides compact yet discriminative representations. We also evaluate PET as a compact topological descriptor on time-series benchmarks. 
%The experiments are designed to validate PET as a compact topological descriptor, rather than to claim universal superiority over existing classification methods or fully expressive directional transforms. 
\end{itemize}

\medskip

The remainder of the paper is organized as follows. Section~\ref{sec:background} introduces the required background on lower-star
filtrations, persistent homology, persistent entropy, as well as existing topological descriptors and topological transforms.
% finite-diagram conventions, and directional topological transforms.
Section~\ref{sec:pet} defines PET and proves its basic geometric properties, including 
% Section 4 establishes stability, directional regularity,
% and finite-sampling approximation results under explicit assumptions. 
%Section~\ref{sec:relation} discusses 
the relation between PET,
PE, PHT, ECT, and standard vectorizations of persistence diagrams. 
Section~\ref{sec:experiments} presents the empirical
validation. Finally, conclusions and future research directions are discussed in Section~\ref{sec:conclusions}.

% The remainder of this paper is organized as follows. Section~\ref{sec:background} introduces the necessary background on simplicial complexes, filtrations, persistent homology, and persistent entropy. Section~\ref{sec:pet} presents the definition of the Persistent Entropy Transform and its theoretical properties. Section~\ref{sec:experiments} contains the experimental evaluation and illustrative examples. Finally, conclusions and future research directions are discussed in Section~\ref{sec:conclusions}.

\section{Background}\label{sec:background}

% explain topology background, persistent homology, perisstent entropy, filtration etc

%This section reviews the mathematical framework underlying the Persistent Entropy Transform (PET). 
In this section, we recall filtrations and directional height functions, persistent
homology and its stability properties, and the definition of PE.
The presentation focuses on the concepts required for the construction and analysis
of PET. For general introductions to TDA, we refer the
reader to
\cite{edelsbrunner2010computational,chazal2021introduction,
edelsbrunner2022computationaltopology}.

\subsection{Simplicial Complexes and Filtrations}

A simplicial complex provides a combinatorial representation of a topological space through simplices of different dimensions. Let $V=\{x_0,\dots,x_n\}\subset\mathbb{R}^d$ be a finite set of vertices. A $k$-simplex is the convex hull of $k+1$ affinely independent vertices of $V$. A simplicial complex $\mathcal{K}$ is a finite collection of simplices satisfying that every face of a simplex in $\mathcal{K}$ also belongs to $\mathcal{K}$, and the intersection of any two simplices is either empty or a common face.

In TDA, simplicial complexes are commonly equipped with filtrations. A filtration is a nested sequence of simplicial complexes
\[
\emptyset=\mathcal{K}_0\subseteq \mathcal{K}_1\subseteq\cdots\subseteq \mathcal{K}_m=\mathcal{K},
\]
encoding the evolution of topological structures across scales. In this work, we focus on lower-star filtrations induced by scalar functions defined on the vertices of a simplicial complex.

\subsection{Directional height functions and lower-star filtrations}

Let $\mathcal{K}$ be a simplicial complex with vertex set $V\subset\mathbb{R}^d$, and let
\[
f:V\to\mathbb{R}
\]
be a scalar function defined on its vertices. The lower-star filtration induced by $f$ is obtained by assigning to each simplex the maximum filtration value among its vertices. This generates a nested family of subcomplexes ordered according to increasing function values. 

In this work, filtrations are induced by directional height functions. Given a unit direction $
v\in S^{d-1}$,
the directional filtration function is defined as
\[
f_v(x)=\langle x,v\rangle,
\]
where $\langle\cdot,\cdot\rangle$ denotes the Euclidean inner product.
For each direction $v$, the function $f_v$ induces a lower-star filtration on $\mathcal{K}$. These directional filtrations provide geometric information about the underlying object from multiple orientations, and form the basis of several important topological transforms. 

The Persistent Homology Transform
(PHT) \cite{Turner2014-ws} assigns to each direction the collection of persistence diagrams associated with the corresponding filtration.
The Euler Characteristic Transform (ECT) \cite{ghrist2018euler, Roell24aDifferentiableECTforshapeclassification,Munch02012025} replaces persistence diagrams by Euler characteristic curves.
Both transforms encode anisotropic geometric information and admit injectivity results under suitable assumptions on
the underlying shape.

\subsection{Persistent Homology, Persistence Diagrams and Barcodes, and Diagram Stability}

Persistent homology is a tool for studying data geometry and connectivity across scales. Using this tool and given a \emph{filtered simplicial complex} (e.g., lower-star filtration), one can effectively compute $k$-dimensional topological features at different scales.
%, tracking features as they appear (birth) and disappear (death). 
Persistence diagrams and barcodes are visual representations used in the study of persistent homology.
%and represent the birth and death (end) of topological features across scales. 
%A $k$-dimensional persistence diagram plots points $(b, e)$ corresponding to the lifespans of $k$-dimensional topological features, while a barcode represents them as a {\it bar} starting at $b$ and ending at $e$. 
Specifically, a $k$-dimensional persistence diagram $D = \{(b_i, e_i) \mid i \in I\}$ satisfies that $b_i \leq e_i$ for all $i \in I$, and $I$ is the index set that identifies the pairs $(b_i,e_i)$ in $D$. Each pair $(b_i,e_i)$ corresponds to a $k$-dimensional topological feature that {\it appears} at time $b_i$ (birth) and {\it disappears} at time $e_i$ (death, i.e., end) as the filtration progresses. 
The associated barcode represents each pair $(b_i,e_i)$ as a {\it bar} starting at $b_i$ and ending at $e_i$.
Both persistence diagrams and barcodes facilitate the understanding of feature persistence in data. Fig. \ref{fig:entropyCalculation} shows a persistence diagram (center) alongside its corresponding barcode (right), illustrating the persistence of topological features in the data.

It is well-known that the number of bars of the barcode associated with the lower-star filtration of a simplicial complex $\K$ is less than or equal to half of the number of vertices of the simplicial complex $\K$ plus 1. Given the lower-star filtration on a simplicial complex $\K$ with vertex set $V$, local critical points are always located at the vertices of $V$ and can be effectively computed from the {\it lower link} of each vertex $v$ of $V$, which consists of the simplices of $\K$ in the closed lower-star that do not belong to the lower-star of $v$. We call $v$ a local critical vertex of index $q$ if its lower link has the reduced homology of the $(q -1)$-sphere. Then, $q$-dimensional persistence diagrams pair critical vertices of index $q$ (when topological features are born) with critical vertices of index $q+1$ (when topological features die). A detailed description can be found in~\cite[Chapter VI]{edelsbrunner2008persistenthomology}.

\begin{rem}
In the 1-dimensional setting considered in our experiments, the number of 0-dimensional persistence intervals is controlled by the number of local minima of the signal since, under genericity assumptions, components are born at local minima and merge at local maxima. By convention, we truncate the death time of the oldest connected component at
$\max f$. Thus, the longest interval is always $[\min f,\max f)$. Figure~\ref{fig:entropyCalculation}
illustrates this construction for a simplicial complex consisting of vertices and
edges, with a function $f$ derived from a sine function. In this example, four
local minima give rise to four 0-dimensional persistence intervals, with the
corresponding component mergers occurring at local maxima. The same persistence
barcode can therefore be represented using seven vertices, corresponding to four
local minima and three local maxima.

\end{rem}

Persistent homology provides a multiscale summary of connected components ($k = 0$), loops ($k = 1$), voids ($k = 2$), and higher-dimensional topological features.

The algebraic foundations of persistent homology were established by Zomorodian and Carlsson~\cite{zomorodian2005computing}, who proved the structure theorem for persistence modules over a field. Efficient computational algorithms and software implementations are now standard within the TDA literature~\cite{Otter2017ComputationPersistentHomology}.

\paragraph{Stability theory.}
%A fundamental property of persistence diagrams is their stability under perturbations of the filtering function. 
Cohen-Steiner et al.~\cite{cohen-steiner2007stability} proved the bottleneck stability inequality

\begin{equation}
d_{\infty}\bigl(\mathrm{D}(h), \mathrm{D}(g)\bigr)
\leq
\|h-g\|_{\infty},
\label{eq:bottleneck_stability}
\end{equation}

for tame functions defined on a fixed simplicial complex.

Subsequent work generalized this result to persistence modules and geometric filtrations~\cite{Chazal2014Stability,chazal2016structure}, establishing stability with respect to interleaving and Gromov--Hausdorff distances. These results provide the theoretical basis for the robustness of persistent homology under geometric and sampling perturbations and constitute the main stability framework used later in the analysis of the persistent entropy.

The stability of persistence-based descriptors has enabled their integration into statistical inference and machine learning pipelines. Applications include topological signal analysis, representation learning, topological regularization, and differentiable architectures~\cite{hofer2017DeepLearningTopologicalSignatures,TopologicalDeepLearning2024,gabrielsson2020Topologylayerformachinelearning,MALYUGINA2023TologicalLossFunctionforImageDenoisisng,ZHANG2025PersistentHomologyCombinedWithMachineLearningforSocialNetworkActivityAnalysis}. 
%The resulting methodological ecosystem motivates the development of descriptors that are simultaneously stable, expressive, and computationally tractable.

\subsection{Existing topological descriptors}

A central challenge in TDA is converting persistence diagrams into representations
compatible with statistical and machine learning pipelines. Since persistence diagrams
are multisets with variable cardinality and non-Euclidean geometry, numerous methods
have been proposed to embed them into functional or vector spaces while preserving
stability and discriminative power.

\paragraph{Functional embeddings.}
Persistence landscapes~\cite{JMLR:v16:bubenik15a} embed persistence diagrams into a
Banach space of piecewise linear functions, enabling the direct use of probabilistic
and statistical tools---means, variances, hypothesis tests---with convergence guarantees
inherited from the Banach space structure. Persistence
images~\cite{JMLR:v18:16-337} provide a finite-dimensional grid-based representation
via kernel density estimation over the persistence plane, combining stability under
perturbations with compatibility with standard machine learning pipelines.

\paragraph{Kernel and algebraic methods.}
Sliced Wasserstein kernels~\cite{pmlr-v70-carriere17a} define similarity measures
between diagrams via optimal transport, while complex-vector
summaries~\cite{DBLP:conf/iciap/FabioF15} encode diagrams through algebraic
transforms of their birth--death coordinates. Both approaches balance expressive power
against computational tractability, but typically produce intermediate- to
high-dimensional representations.

\paragraph{Persistent Entropy}

Persistent entropy, introduced in~\cite{chintakunta2015entropybarcode}, provides an information-theoretic summary of persistence diagrams by quantifying the distribution of persistence interval lengths.

\begin{defn}[Persistent entropy]\label{de:pe}
    Given a persistence diagram
\[
D=\{(b_i,d_i)\mid i\in I\},
\]
let
\(
\ell_i=d_i-b_i
\)
denote the persistence length of the $i$-th interval, and define
\(
L=\sum_{i\in I}\ell_i.
\)
\\
The persistent entropy of $D$ is defined as
\[
PE(D)
=
-\sum_{i\in I} p_i\log(p_i),
\]
where
\(
p_i=\frac{\ell_i}{L},
\) and log denotes the natural logarithm.

\end{defn}

\begin{figure}[ht!]
     \centering
     \includegraphics[width=0.7\textwidth
     ]{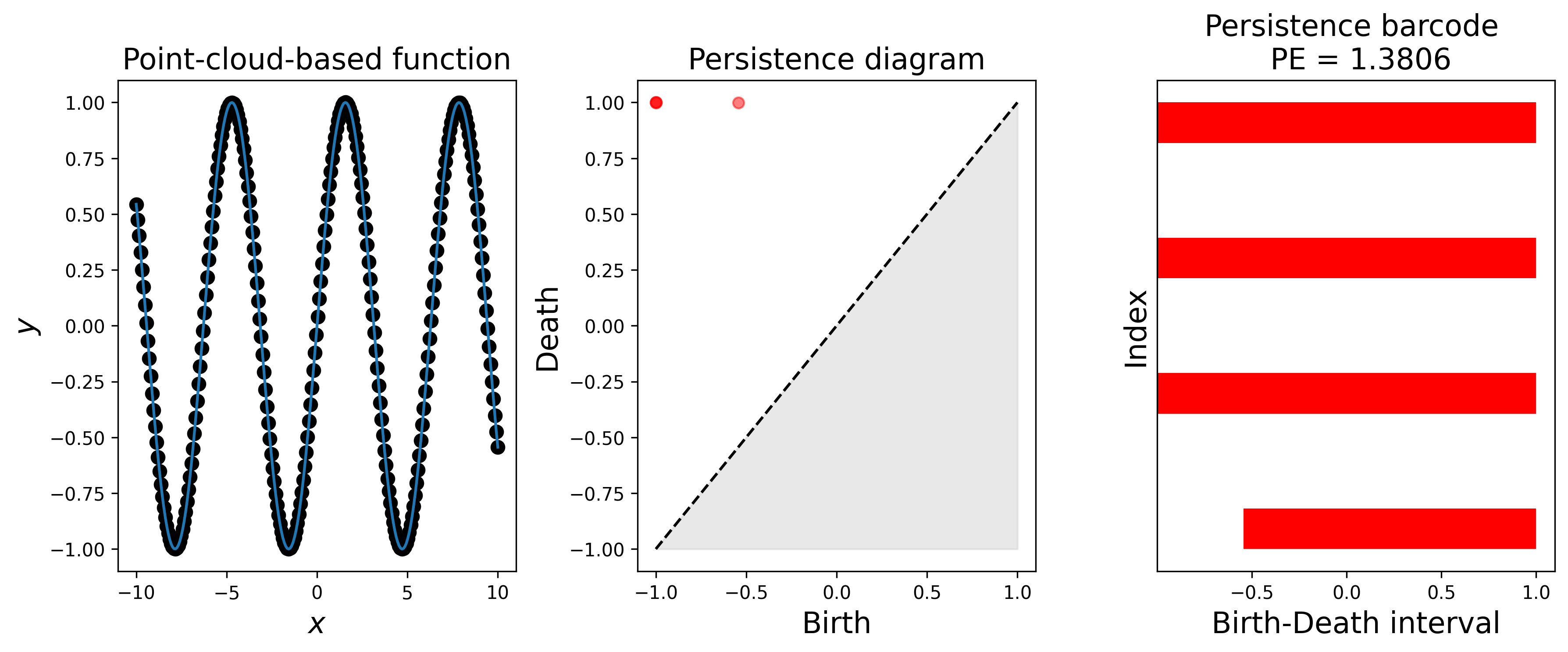}
    \caption{The figure shows a simplicial complex with 250 vertices (point-cloud-based function with 250 points) representing the sine function (left), its persistence diagram (middle), barcode, and persistent entropy (right), all from the lower-star filtration.} \label{fig:entropyCalculation}
\end{figure}

As an illustration, the persistent entropy computed in the example shown in Fig.~\ref{fig:entropyCalculation} (left) is 1.3806.

%Persistent entropy reaches its maximum value when all persistence intervals have equal length, in which case it is given by $\PE = \ln(\#I)$, where $\#I$ denotes the cardinality of $I$ (i.e., the number of points in the persistence diagram $D$), and decreases as the distribution becomes more concentrated around a small number of dominant intervals.

Persistent entropy is precisely the Shannon entropy of the normalized persistence length distribution. It attains its maximal value when all persistence intervals have equal length and decreases as the distribution becomes concentrated around a smaller number of dominant features. Consequently, persistent entropy provides a compact quantitative measure of the dispersion of topological persistence within the diagram.
Besides, under standard finiteness and boundedness assumptions, persistent entropy is stable with respect to small perturbations of the persistence diagram. This last property
%stability of persistence diagrams under perturbations of the input data 
provides robustness guarantees for entropy-based summaries of persistent homology~\cite{atienza2019persistent,atienza2020stabilitypersistententropy}.

% Atienza et al.~\cite{atienza2019persistent,atienza2020stabilitypersistententropy} proved that persistent entropy is continuous with respect to perturbations of the persistence diagram. More precisely, for persistence diagrams $D$ and $D'$,
% \begin{equation}
% \bigl|\mathrm{PE}(D) - \mathrm{PE}(D')\bigr|
% \leq
% C(D, D')\, d_{\infty}(D, D'),
% \end{equation}
% where the constant $C(D, D')$ depends on diagram cardinalities and total persistence. Combined with bottleneck stability~\cite{cohen-steiner2007stability}, this result provides explicit robustness guarantees for entropy-based summaries of persistent homology.

\paragraph{Persistent Entropy-based Summaries.}
Persistent entropy~\cite{chintakunta2015entropybarcode} occupies a distinctive
position in this landscape due to its minimal dimensionality, information-theoretic
interpretability, and $O(n \log n)$ computational cost. By mapping a persistence
diagram to a single scalar via the Shannon entropy of normalized persistence lengths
(Def.~\ref{de:pe}), it sacrifices representational richness in exchange for
compactness and ease of statistical manipulation. As established in
Section~\ref{sec:background} and exploited in Section~\ref{sec:pet}, this
compactness is accompanied by rigorous  stability
guarantees~\cite{atienza2020stabilitypersistententropy}.

\paragraph{Limitations of persistent entropy-based summaries.}
%Despite its compactness and stability, 
Persistent entropy compresses the entire persistence diagram into a single scalar value. As a consequence, the descriptor does not explicitly retain directional or anisotropic geometric information associated with the filtration. Different geometric structures may therefore produce identical entropy values whenever their persistence-length distributions are sufficiently similar. For example, the ellipse and the circle shown in Fig.~\ref{fig:motivation} yield the same persistent entropy, whereas the ellipse and a rotated version of it may yield different entropy values under directional filtrations. These examples illustrate that persistent entropy alone cannot reliably distinguish between geometrically different shapes, while its dependence on the chosen filtration may also hinder invariance to orientation. Consequently, a single scalar entropy value provides only a coarse characterization of the underlying topology, so it should be regarded as a compact global descriptor rather than a complete shape representation, motivating the incorporation of directional information to better capture anisotropic geometric features. The directional limitation of
persistent entropy, its insensitivity to anisotropy and orientation, is precisely the
gap that PET is designed to fill.

\subsection{Directional Topological Transforms}
\label{sec:related:directional}

A parallel line of research constructs topological signatures by evaluating a fixed
descriptor along a continuum of directional filtrations, yielding functional
representations over $S^{d-1}$ that encode the anisotropic geometry of the input shape.

\paragraph{Persistent Homology Transform.}
The PHT~\cite{Turner2014-ws} assigns to each direction $v \in S^{d-1}$ the tuple of
persistence diagrams arising from the directional lower-star filtration $f_v(x) =
\langle x, v \rangle$, one diagram per homology degree. Turner et al \cite{Turner2014-ws}\ proved
that PHT is injective on embedded simplicial complexes in $\mathbb{R}^d$: two distinct
shapes yield distinct PHT signatures, so the transform fully characterizes geometry up
to the ambient embedding. This injectivity makes PHT a gold standard for shape
discrimination, but comes at the cost of storing and comparing one full persistence
diagram per direction---an $O(N n^\omega)$ pipeline where $N$ is the number of
discretized directions and $\omega$ is the matrix multiplication exponent.

\paragraph{Euler Characteristic Transform.}
The ECT~\cite{ghrist2018euler,Curry2022-jl} replaces persistence diagrams with the
integer-valued Euler characteristic function $\chi(\mathcal{K}_t)$, computed at each
filtration parameter $t$ for each direction $v$. Curry et al.~\cite{Curry2022-jl}
proved that ECT achieves injectivity with a \emph{finite} number of directions under
mild geometric conditions on the input complex, a stronger result than that currently
available for PHT. The per-direction output is a scalar function rather than a diagram,
reducing storage and comparison costs relative to PHT, but the representation remains
functional and high-dimensional when the full profile over $S^{d-1}$ is retained.

\section{Persistent Entropy Transform}\label{sec:pet}

The Persistent Entropy Transform  combines persistent entropy with directional filtrations, yielding a compact functional descriptor over $S^{d-1}$.
This section presents the formal development of PET. We give the definition, establish its basic %geometric
properties, 
%prove a quantitative stability theorem, 
and discuss the 
homology
degree and 
discretization choices that arise in practice.

\subsection{Definition, Algorithm and Relation to other Transforms}

We begin by introducing the persistent entropy transform through its formal definition and computational algorithm. We then discuss its relationship with other related transforms.

\begin{defn}[Persistent Entropy Transform]
Let $\mathcal{K}$ be a simplicial complex with vertex set $V\subset\mathbb{R}^d$.
For each direction $v \in S^{d-1}$, let $f_v :V \to \mathbb{R}$ be the directional height function
\[
f_v(x) = \langle x, v \rangle,
\]
and let $\mathrm{D}_k(f_v)$ denote the $k$-dimensional persistence diagram of the lower-star filtration induced by $f_v$.

The Persistent Entropy Transform of $\K$ in homology degree $k$ is the function
\[
\mathrm{PET}^{(k)}_\K : S^{d-1} \to \mathbb{R},
\qquad
v \mapsto \mathrm{PE}\bigl(\mathrm{D}_k(f_v)\bigr).
\]
\end{defn}

Thus, for a fixed homology degree $k$, $\mathrm{PET}^{(k)}_\K$ is a scalar-valued
function whose domain is the space of directions $S^{d-1}$. When the homology degree is fixed or clear from context (in particular, when $k=0$), we write
$\mathrm{PET}_\K$ for brevity and simply PET when the underlying simplicial complex is also fixed or clear from the context.

In practice, PET is evaluated on a set of $N$ directions
$\{v_1,\ldots,v_N\}\subset S^{d-1}$, for some $N\in\mathbb{N}$, yielding the discretized representation
\[
\mathrm{PET}_\K^{(k)} =
\left(
\mathrm{PE}(\mathrm{D}_k(f_{v_1})),
\ldots,
\mathrm{PE}(\mathrm{D}_k(f_{v_N}))
\right)
\in \mathbb{R}^N.
\]
Intuitively, each evaluation $PE(\mathrm{D}_k(f_{v}))$ measures the information-theoretic complexity of the topological structures of $\K$ as seen from direction $v\in S^{d-1}$: a high value indicates that topological features at that viewpoint have broadly distributed lifetimes, while a low value indicates dominance by a single long-lived feature. The full function over $S^{d-1}$ therefore encodes how this complexity varies with orientation, capturing the topology of $\K$ as a compact scalar-valued profile.

The computational procedure for obtaining this discretized representation is summarized in Algorithm~\ref{alg:pet}.
\begin{algorithm}[h!]
\caption{Algorithm 1 Computation of $PET_K^k$}
\label{alg:pet}
\begin{algorithmic}
\Require Simplicial complex \( \K \) with vertex set $V\in\R^d$, directions \( \{v_1, \ldots, v_N\} \subset S^{d-1} \), homology degree \( k \)
\Ensure Vector in $\mathbb{R}^N$
\For{\( j = 1 \) to \( N \)}
\State Compute \( f_{v_j}(x) = \langle x, v_j \rangle \) for all \( x \in V \)
\State Build lower-star filtration of \( \K \) induced by \( f_{v_j} \)
\State Compute \( D_{\text{k}}(f_{v_j}) \) 
\State \( \pe_j \leftarrow \PE(D_{\text{k}}(f_{v_j})) \) (Def.~\ref{de:pe})
\EndFor
\State return \( (\pe_1, \ldots, \pe_N) \)
\end{algorithmic}
\end{algorithm}

For a fixed homology degree $k$, the length of the discretized PET
representation is determined by the number $N$ of sampled directions.
%, rather than by the dimension $d$ of the embedding space orthe number of vertices of $\K$.

If several homology degrees are considered, the corresponding transforms can
be collected as
\[
\left(
\mathrm{PET}^{(0)}_\K,\ldots,\mathrm{PET}^{(d-1)}_\K
\right),
\]
where each component is a function on $S^{d-1}$ (or, after sampling, a vector in
$\mathbb{R}^N$).

\begin{rem}[Relationship to PHT and ECT]
The full PET can be seen as a functional compression of the PHT: where PHT stores the complete diagram $\mathrm{D}_k(f_v)$ for each $(v,k)$, PET retains only its Shannon entropy. Analogously, ECT stores the Euler characteristic $\chi(\K_t^v)$ as a function of $t$ for each $v$, 
where
$
K_t^v$ is the 
sublevel-set complex 
$\{\sigma\in \K:\max_{x\in\sigma} f_v(x)\le t\}
$,
whereas PET integrates this information into a single scalar via entropy. This compression is lossy (PET has not been proven to be injective in general) but yields a descriptor that is a vector in $\mathbb{R}^N$ after discretization of $S^{d-1}$, directly compatible with downstream pipelines.
\end{rem}

\subsection{Basic Geometric Properties}

We establish four properties of PET: boundedness, translation invariance, rotational equivariance, and behavior under uniform scaling. Together, they characterize how PET responds to the most common geometric transformations encountered in shape analysis.

\begin{prop}[Boundedness]
\label{prop:bound}
Let \( n_{v}\) denote the cardinality of
      $D_{k}(f_{v})$ for any \( v \in S^{d-1} \).
 Let 
    $\PET_\K^{(k)}= (\pe_1, \ldots, \pe_N)$ be the output of Alg.\ref{alg:pet}, where $\pe_j=D_{k}(f_{v_j})$ for some 
     \( v_j \in S^{d-1} \).      
    Then,  for all $j$,
\[ 0 \leq \pe_j \leq \log n_v\,. \]
The lower bound is attained when the barcode associated to the diagram $D_{k}(f_{v_j})$ contains a single interval of nonzero length; the upper bound is attained when all intervals of 
the barcode associated to $D_{k}(f_{v_j})$ have equal length.
\end{prop} 

\begin{proof}Both bounds follow from the standard properties of Shannon entropy: \( \pe(p) \geq 0 \) for any probability distribution \( p \), and \( \PE(p) \leq \log m \) for a distribution over \( m \) atoms, with equality if and only if \( p \) is uniform. Applying these bounds to the distribution \( \{p_i\}_{i \in I} \) of normalized persistence lengths (Def.~\ref{de:pe}) gives the result.
\end{proof}

\begin{prop}[Translation invariance]
\label{prop:trans}
    Let $\K$ be a simplicial complex with vertex set $V$.
Let \( a \in \mathbb{R}^d \) and let \( \K + a \) the simplicial complex  with vertex set 
\(V+a=\{v + a : v \in V\} \), denote the translate of \(\K \). Then
\[ \PET_{\K+a} = \PET_\K. \]
\end{prop}
\begin{proof}The height function on the translated complex satisfies \( f_v(x + a) = \langle x + a, v \rangle = f_v(x) + \langle a, v \rangle \). Hence, all birth and death times in \( D_{k}(f_v) \) are shifted by the constant \( \langle a, v \rangle \), leaving every persistence length \( \ell_i = d_i - b_i \) unchanged. Since PE depends only on the normalized lengths \( \{p_i\} \) (see Def.~\ref{de:pe}), the result follows.
\end{proof}

\begin{prop}[Rotational equivariance]
\label{prop:rot}
Let $\K$ be a simplicial complex with vertex set $V$.
Let    $\PET_\K(v)$ denote the value $\PE(D_k(f_v))$.
Let \( R \in SO(d) \) be a rotation and let \( R \K \) be the simplicial complex with vertex set \(RV=\{R x : x \in V\} \).
Then
\[ \PET_{R \K} (v)= \PET_K(R^\top v) \quad \text{for all} \quad v \in S^{d-1}. \]
\end{prop}
\begin{proof}For any vertex \( x \in V \), \( \langle R x, v \rangle = \langle x, R^\top v \rangle = f_{R^\top v}(x) \). The directional filtration of \( R \K \) in direction \( v \) therefore coincides with the directional filtration of \( \K \) in direction \( R^\top v \), giving \( D_{k}(f_{v}) = D_{k}(f_{R^\top v}) \) and hence the stated equality.
\end{proof}

\begin{prop}[Behavior under uniform scaling]
\label{prop:scal}
Let $\K$ be a simplicial complex with vertex set $V$.
Let \( \lambda > 0 \) and let \( \lambda V = \{\lambda x : x \in V\} \). Then
\[ \PET_{\lambda \K} = \PET_\K  \]
\end{prop}
\begin{proof}Scaling by \( \lambda \) transforms the height function as \( f_v(\lambda x) = \lambda f_v(x) \), so every birth and death time is multiplied by \( \lambda \). The persistence lengths become \( \ell_i' = \lambda \ell_i \), and the total persistence becomes \( L' = \lambda L \). The normalized weights are therefore \( p_i' = \lambda \ell_i/\lambda L = \ell_i/L = p_i \), leaving  PE invariant.
\end{proof}

By Propositions~\ref{prop:bound}, \ref{prop:scal}, and \ref{prop:trans}, the PET profile is invariant under affine transformations of the space.

\begin{theorem}  
[Isometry invariance]
   For any affine transformation \( \phi:\R^d\to\R^d \),
\[ \PET_{\phi\K} =  \PET_\K \quad \mbox{up to the equivariant reparametrization of \( S^{d-1} \)}. \]
\end{theorem} 

%In shape analysis applications where orientation is not canonically fixed, one therefore compares shapes by comparing their PET profiles up to rotation of the sphere, e.g. via rotationally invariant norms such as \( \| \text{PET}_{K_1} - \text{PET}_{K_2} \circ R \|_2(S^{d-1}) \) minimized over \( R \in SO(d) \). Proposition 4 further shows that PET is invariant under uniform rescaling, so shape comparisons are automatically scale-independent.

We end this section by noting that the robustness guarantees established for entropy-based summaries of persistent homology~\cite{atienza2019persistent,atienza2020stabilitypersistententropy} naturally extend to PET. In particular, they ensure that small perturbations in the input data result in only small changes in the PET profile.

\subsection{Practical Remarks}

\textbf{Choice of homology degree.} In most shape analysis applications, \( k = 0 \) (connected components) and \( k = 1 \) (loops) provide the most discriminative information. In \( \mathbb{R}^2 \), the full PET profile is the pair \( (\PET^{(0)}, \PET^{(1)}) \); in \( \mathbb{R}^3 \), the triple \( (\PET^{(0)}, \PET^{(1)}, \PET^{(2)}) \). For signal analysis applications in \( \mathbb{R}^1 \), only \(\PET^{(0)} \) is nontrivial. 
%Unless stated otherwise, experiments in Section 5 report results for \( k \in \{0, 1\} \) separately and in concatenation.

\textbf{Discretization of \( S^{d-1} \).} In practice, \( S^{d-1} \) is replaced by a finite set of \( N \) directions \( \{v_1, \ldots, v_N\} \), and \( \PET \) is approximated by the vector \( (\PE(D_{k}(f(v_1))), \ldots, \PE(D_{k}(f(v_N)))) \in \mathbb{R}^N \). For \( d = 2 \), uniform angular sampling \( v_j = (\cos(2\pi j/N), \sin(2\pi j/N)) \) is standard. 

For \( d = 3 \), a Fibonacci lattice or a subdivision of the icosahedron provides near-uniform coverage. For \( d \geq 4 \), quasi-random sequences on \( S^{d-1} \) (e.g., randomized projections) could be used. 
The approximation error introduced by discretization is controlled by the continuity of PET as a function of the direction $v$. Since the filtering functions $f_v$ depend continuously on $v$, the stability of persistence diagrams implies that $D_k(f_v)$, and hence $\PET(v)$, vary continuously with $v$. Therefore, sufficiently dense sampling of $S^{d-1}$ yields an accurate approximation of the PET profile.

%The approximation error introduced by discretization is controlled by the modulus of continuity of \( \PET \) as a function of \( v \), which is finite whenever the persistence diagrams \( D_{k}(f_v) \) vary continuously with \( v \)—a condition satisfied generically for embedded simplicial complexes. Sensitivity to \( N \) is analyzed empirically in Section 5.

%Continuity of \( v \rightarrow PET_k(v) \) follows from the fact that \( v \rightarrow f(v) \) is continuous in \( \| \cdot \|_\infty \) (since \( |f(v(x)) - f(v'(x))| = | \langle x, v - v' \rangle | \leq \text{diam}(K) \|v - v'\| \)), combined with the bottleneck stability of Theorem 1.

\textbf{Normalization.} For scale and translation invariant comparisons, one may center the vertex set (subtract the centroid) and normalize by the diameter before computing PET. Proposition \ref{prop:trans} and \ref{prop:scal} guarantee that the resulting descriptor is equivalent to computing PET on the normalized complex directly.e

As an illustration, Fig.~\ref{fig:pet}  shows the persistent entropy values obtained from the persistence diagrams induced by the directional filtrations corresponding to the 64 uniformly sampled orientations shown in Fig \ref{fig:directions}. 
Figs.~\ref{fig:rot} and \ref{fig:trans} illustrate  the translation and rotation  invariance property of the PET profile.

\begin{figure}[ht!]
     \centering
     \includegraphics[width=0.7\textwidth
     ]{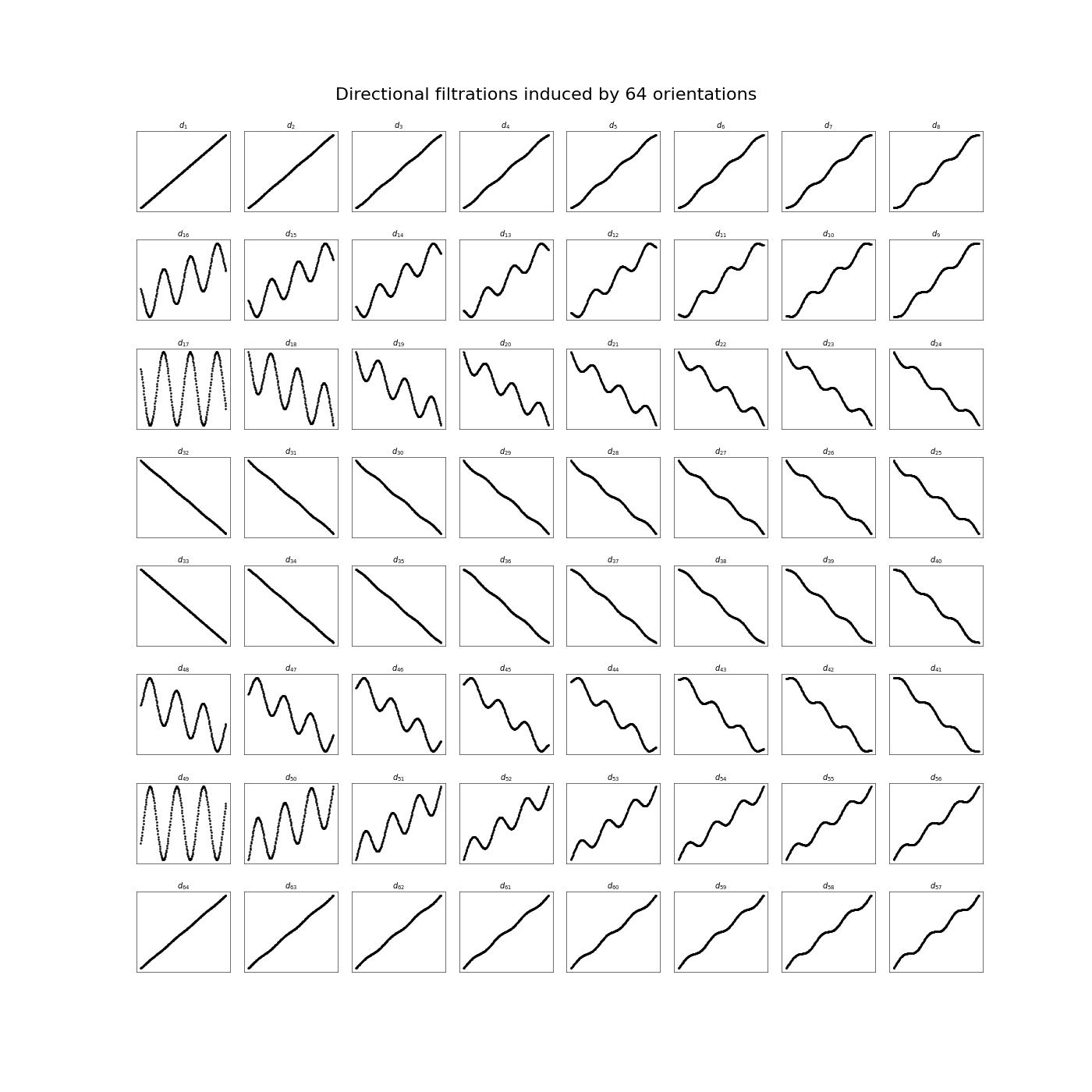}
    \caption{Directional filtrations of the sine function of Fig \ref{fig:entropyCalculation} induced by multiple orientations. Each subplot corresponds to the values of the lower-star filtration obtained from the directional height function associated with a specific direction $v\in S^{d-1}$. We consider 64 uniformly sampled orientations. Each point represents the filtration value of a simplex when it appears in the filtration. } \label{fig:directions}
\end{figure}

\begin{figure}[ht!]
     \centering
     \includegraphics[width=0.6\textwidth
     ]{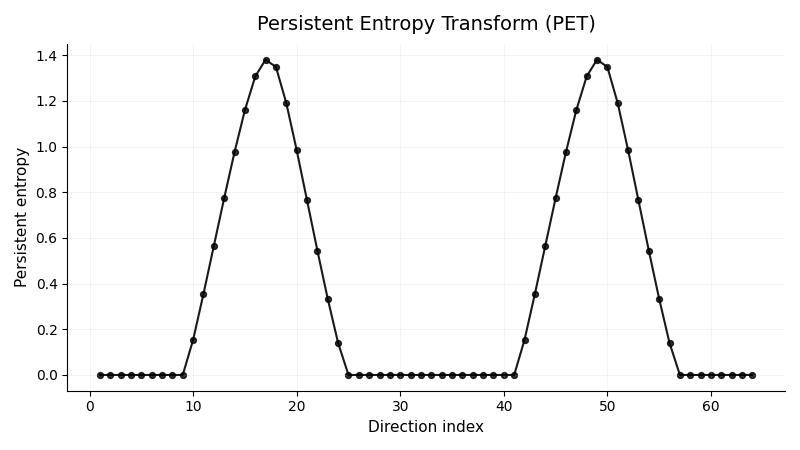}
    \caption{Persistent Entropy Transform (PET): The horizontal axis corresponds to the direction index, while the vertical axis represents the persistent entropy, which quantifies the topological complexity observed from each direction.} \label{fig:pet}
\end{figure}

\begin{figure}[ht!]
     \centering
     \includegraphics[width=\textwidth
     ]{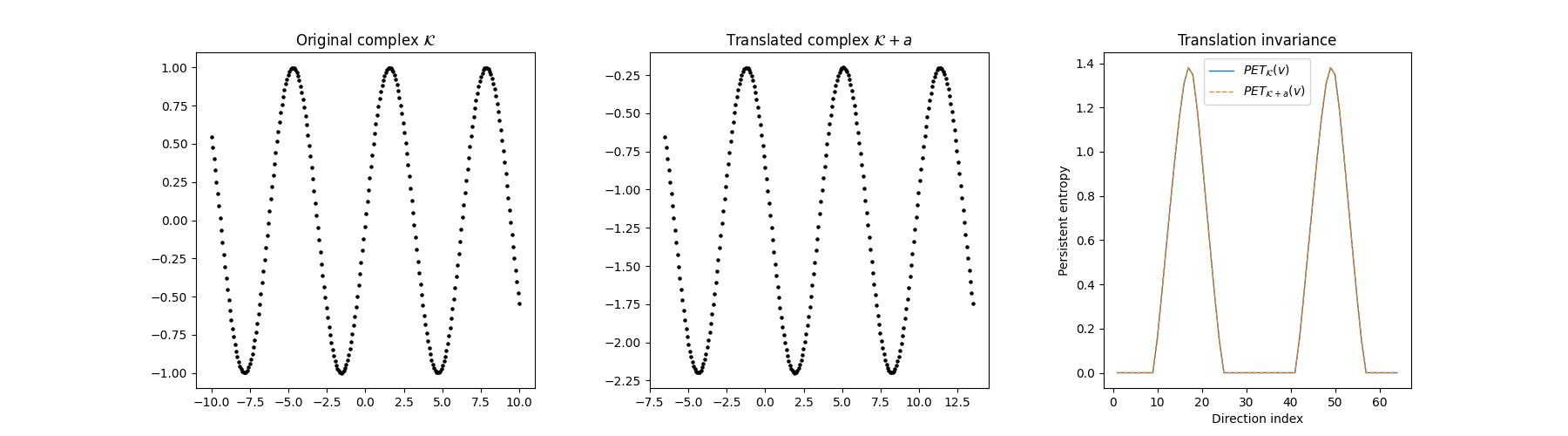}
    \caption{(Left) A point cloud representing a 1D signal and its translated version (Middle). (Right) The PET computed for the original and the translated complex coincide, showing invariance under translations.} \label{fig:trans}
\end{figure}

\begin{figure}[ht!]
     \centering
     \includegraphics[width=\textwidth
     ]{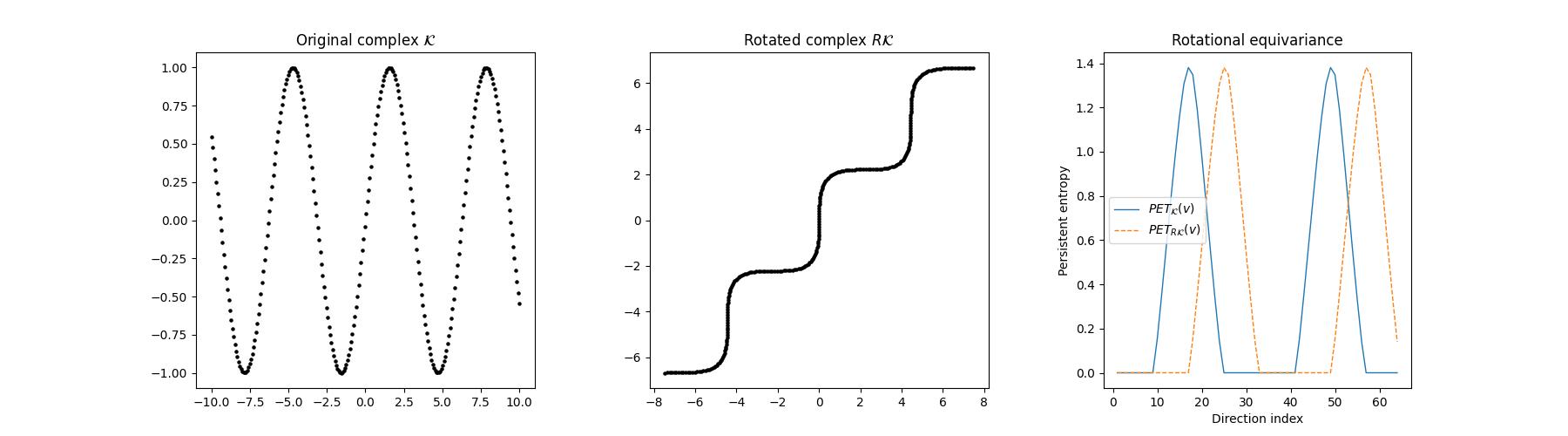}
    \caption{(Left) A point cloud and its rotated version (Middle). (Right) The PET of the rotated complex matches the PET of the original complex, being translated along the x-axis, confirming rotational equivariance.} \label{fig:rot}
\end{figure}

\subsection{Relationship with existing topological descriptors}\label{aubsec:relation}

Standard vectorization methods---persistence landscapes~\cite{JMLR:v16:bubenik15a},
persistence images~\cite{JMLR:v18:16-337}, sliced Wasserstein
kernels~\cite{pmlr-v70-carriere17a}, and complex-vector
summaries~\cite{DBLP:conf/iciap/FabioF15}---embed persistence diagrams into Euclidean
or functional spaces.
Among these, persistent entropy~\cite{chintakunta2015entropybarcode} occupies a
distinctive position due to its minimal dimensionality,
%and $O(n\log n)$ cost,
making it the natural scalar baseline for PET.

Besides, PET shares the directional philosophy of PHT and ECT---evaluating a topological
descriptor along the family of height-function filtrations $\{f_v\}_{v \in S^{d-1}}$---
but replaces the per-direction summary with the scalar persistent entropy $\PE(D_k(f_v))$.
This substitution has three concrete consequences. First, the per-direction output is a
single real number rather than a function or diagram, so the full PET signature is a
vector in $\mathbb{R}^N$ after discretization of $S^{d-1}$, directly compatible with
any downstream statistical or machine learning method. Second, the computational cost
reduces to $O(N n \log n)$, matching ECT and improving on PHT. Third, PET inherits
the robustness of persistent entropy (Section~\ref{sec:pet}).
%rather than the set-valued stability of PHT; this is a weaker property, but one that comes with explicit constants and direct error bounds on the scalar output.

% -----------------------------------------------------------------------
%\subsection{Formal Comparison}
%\label{sec:related:comparison}

% Table~\ref{tab:transforms} (Section~\ref{tab:transforms}) summarizes the four descriptors
% along the axes of output type, stability, injectivity, and computational complexity.
% We expand on two dimensions that are particularly consequential for applications.

%\paragraph{Injectivity.}
Moreover, PHT and ECT are injective under their respective conditions; PET is not known to be
injective in general.
%and we conjecture that it is not. Concretely, two shapes that differ only in the \emph{relative ordering} of their persistence intervals across directions, but whose entropy profiles coincide, would yield identical PET signatures. Constructing an explicit counterexample, or establishing injectivity under additional geometric assumptions, remains an open problem that 
We identify this as a primary direction
for future work (see Section~\ref{sec:conclusions}).

%\paragraph{Compactness and statistical tractability.}
Now, observe that the output of PET after discretization with $N$ directions is a vector
$(\PE(D_k(f_{v_1}), \dots, \PE(D_k(f_{v_N})) \in \mathbb{R}^N$, which can be
used directly as a feature vector in any kernel method, neural network, or statistical
test without additional post-processing. In contrast, PHT requires comparing persistence
diagrams via bottleneck or Wasserstein distances---computationally expensive operations
that complicate statistical inference---and ECT requires comparing scalar functions over
$S^{d-1}$, typically via $L^2$ or $L^\infty$ norms after further discretization. This
difference in downstream tractability is a practical advantage of PET in large-scale or
time-constrained settings, and is confirmed empirically in the timing experiments of
Section~\ref{sec:experiments}.

\medskip
\noindent\textbf{Summary.} PET is not intended to replace PHT or ECT in applications
where full geometric characterization or injectivity is required. Rather, it offers a
principled and computationally efficient alternative when compactness, interpretability,
and direct statistical compatibility are the primary constraints. It extends persistent
entropy from a global scalar descriptor to a directional functional representation,
and extends ECT from the Euler characteristic to an information-theoretic entropy
summary, occupying a well-defined and previously unoccupied position in the landscape
of topological descriptors. PET is related to ECT by analogy rather than generalization: both
evaluate a scalar topological summary at each direction, but ECT
uses the Euler characteristic and PET uses persistent entropy. The
two are complementary in that ECT is injective and PET is more
compact and statistically tractable.

\section{Experiments}\label{sec:experiments}

This section evaluates PET as a compact directional descriptor. 
%The experiments are not designed to establish PET as a universal classifier or as a replacement for injective directional transforms. Instead, 
The experiments test the empirical consequences
of the theory developed above: directional sensitivity, consistency with orthogonal equivariance, behaviour under
indexed perturbations, dependence on directional sampling, and usability as a finite-dimensional representation on
real time-series benchmarks.

\subsection{Experimental protocol}

All experiments use the finite-diagram convention introduced in Section \ref{sec:background}: persistent entropy is computed only from
finite persistence intervals, and empty finite diagrams are assigned entropy zero. PET is computed in homology degree \( k = 0 \), using directional lower-star filtrations on finite embedded one-dimensional simplicial complexes. We focus on \( H_0 \) because it provides a simple and reproducible setting for evaluating the directional behaviour of PET across both synthetic curves and time-series polylines. This choice should be understood as an initial validation setting rather than as a claim that higher-dimensional PET components are unnecessary.

For planar synthetic shapes, each object is represented as a polygonal curve. Vertices are sampled points in \( \mathbb{R}^2 \), and edges connect consecutive vertices. Closed curves are represented by adding an edge between the last and first sampled vertices. The circle, ellipse, and rotated ellipse are centred before PET computation in order to make the geometric comparison independent of translations, consistently with Proposition 2. If scale normalization is applied, it is performed before computing all descriptors and is applied identically to PET and to the non-directional baselines.

For time-series data, each signal \( (s_1, \ldots, s_T) \) is represented as a planar polyline with vertices \( x_i = (t_i, s_i) \in \mathbb{R}^2 \), \( i = 1, \ldots, T \), where the time coordinate \( t_i \) is rescaled to \( [0, 1] \). Consecutive vertices are connected by edges. This representation allows synthetic shapes and time-series signals to be processed using the same directional lower-star framework.

In classification experiments, all preprocessing choices are fitted on the training set only and then applied unchanged to the test set.

For a finite set of sampled directions
$\{v_1, \ldots, v_N\} \subset S^1$, the sampled PET feature vector is defined as
\[
\mathrm{PET}_\K = 
\left(
\PE(D_0(f_{v_1})), \ldots, \PE(D_0(f_{v_N}))
\right)
\in \mathbb{R}^N,
\]
where $\PE(D_0(f_{v_j}))$ is computed from the $0$-dimensional persistence
diagram induced by the lower-star filtration in direction $v_j$.

\begin{table}[h]
\centering
\caption{Summary of the experimental protocol. The placeholders indicate values that must be replaced by the exact settings used in the final reproducible implementation.}
\begin{tabular}{lllll}
\hline
Experiment                    & Representations           & Homology & Directions & Main metric                               \\ \hline
Synthetic shapes              & Polygonal curve           & $H_0$     &      64      & Directional variability                   \\ 
Noise perturbations           & Perturbed polygonal curve & $H_0$     &      64      & PET distance                              \\ 
Sampling analysis             & Polygonal curve           & $H_0$     &      Variable      & Approximation error                       \\ 
TwoLeadECG \& MITBIH Datasets & Planar time-series curve  & $H_0$     &     64       & Classification metrics \\ \hline
\end{tabular}
\label{tab:protocol}
\end{table}

Directions are sampled uniformly on \( S^1 \), unless explicitly stated otherwise. For classification experiments, PET features and baseline vector features are standardized using training-set statistics only. 

We include the following baselines when applicable:

\begin{itemize}
    \item Raw vector representation: original sampled coordinates or time-series values used directly as features.
    \item Persistent entropy: a non-directional scalar baseline obtained by applying persistent entropy to a fixed reference filtration. For synthetic shapes, this baseline is computed from a single lower-star filtration along a fixed direction as a non-directional baseline. Unlike PET, this quantity depends on a single observation direction and therefore may change under rotations of the shape.
    \item Directional PET: the proposed descriptor obtained by evaluating PET over the set $\{v_1, \ldots, v_N\}\subset S^1$.
\end{itemize}

All randomized experiments use fixed random seeds. Perturbation experiments are repeated over multiple independent trials, and results are reported as mean and standard deviation. The persistent homology backend, number of sampled vertices, number of sampled directions, perturbation levels, number of repetitions, and classifier hyperparameters must be specified in the final implementation. Table \ref{tab:protocol} summarizes the experimental settings used in this section.

\subsection{Directional Geometric Sensitivity}

The first experiment evaluates whether PET captures directional variability on simple planar shapes. We consider a circle, an ellipse, and a rotated ellipse. These examples are deliberately simple and are intended to verify that PET reflects the geometric behaviour expected from isotropic and anisotropic objects. For each shape, PET is computed over a uniformly sampled set $\{v_1, \ldots, v_N\}\subset S^1$. Owing to its rotational symmetry, the circle is expected to produce a nearly uniform directional response, whereas the ellipse should exhibit direction-dependent variations. Rotating the ellipse should preserve the overall behaviour of the transform while shifting its directional response.

To quantify directional sensitivity, we compute the empirical directional range
\[
\mathrm{Range}(\K)=
\max_{v_j}
\PE(D_0(f_{v_j}))
%\mathrm{PET}_0(K(v_j))
-
\min_{v_j}
\PE(D_0(f_{v_j}))
%\mathrm{PET}_0(K(v_j)),
\]
and the empirical directional variance
\[
\mathrm{Var}
%_{V_N}
(\K)
=
\frac{1}{N}
\sum_{j=1}^{N}
\left(
\PE(D_0(f_{v_j}))
%\mathrm{PET}_0(K(v_j))
-
\overline{
\PE(D_0(f_{v_j}))
%\mathrm{PET}_0(K)
}
\right)^2,
\]
where
\(
\overline{
\PE(D_0(f_{v_j}))
%\mathrm{PET}_0(K)
}
=
\frac{1}{N}
\sum_{j=1}^{N}
\PE(D_0(f_{v_j}))
%\mathrm{PET}_0(K(v_j))
\,.
\)

We also report the persistent entropy as a non-directional baseline. Unlike PET, this scalar descriptor summarizes the filtration using a single value and therefore cannot explicitly encode directional behaviour.

\begin{table}[h]
\centering
\caption{Directional sensitivity of PET on synthetic shapes. PE is a non-directional scalar baseline. Directional
range and variance quantify the variability of PET over sampled directions.}
\begin{tabular}{lllll}
\hline
Shape                    & PE           & Mean PET & Directional range & Directional variance                               \\ \hline
Cyrcle              & 0.6365           & 0.5163     &     0.6931       & 0.049                   \\ 
Ellipse           &  0.6365 & 0.4678 & 0.6931 & 0.069                                          \\ 
Rotated ellipse             &  0.1989 & 0.4678 & 0.6931 & 0.069                    \\ 
 \hline
\end{tabular}
\label{tab:directional_sensitivity}
\end{table}

Figure \ref{fig:motivation} shows the PET computed for a circle, an ellipse, and a rotated ellipse using directional lower-star filtrations sampled uniformly over $S^1$. The circle produces an approximately constant PET, reflecting its rotational symmetry. In contrast, the ellipse exhibits a periodic PET profile associated with its anisotropic geometry. Rotating the ellipse results in a phase shift of the transform, illustrating the rotational equivariance of PET.

Table~\ref{tab:directional_sensitivity} summarizes the numerical results. The persistent entropy computed from a single direction changes after rotating the ellipse, illustrating its dependence on the chosen filtration direction. In contrast, the mean PET and the directional variance remain unchanged for the ellipse and its rotated version, indicating that PET preserves the overall directional behaviour of the shape while encoding the change as a reparameterization of the directional profile. Moreover, the circle exhibits the smallest directional variance, reflecting its higher degree of rotational symmetry compared with the anisotropic ellipses.

\subsection{Robustness under Noise Perturbations}

The second experiment evaluates how PET changes under controlled indexed vertex perturbations. Starting from a fixed ellipse represented as a polygonal curve, we add independent Gaussian perturbations to the vertex coordinates while preserving the connectivity of the polygon. This perturbation protocol matches the indexed-vertex setting considered previously in Section \ref{sec:pet}. For a noise level $\mu$, the perturbed vertices are generated as
\[
\tilde{x}_i = x_i + \eta_i,
\qquad
\eta_i \sim \mathcal{N}(0,\mu^2 I_2).
\]

For each value of $\mu$, the experiment is repeated $R=50$ times. We compute the PET distance
\[
d_{\PET}(\K,\tilde \K)
=
\left\|
\PET_\K^{(0)}
%\mathrm{PET}_0(K,V_N)
-
\PET_{\tilde{\K}}^{(0)}
%\mathrm{PET}_0(\tilde K,V_N)
\right\|_2,
\]
together with the maximum indexed vertex perturbation
\[
\delta
=
\max_i
\|x_i-\tilde{x}_i\|.
\]

The quantity $\delta$ corresponds to the perturbation scale appearing in the indexed-vertex stability result of Section \ref{sec:pet}. Rather than providing a proof of the stability theorem, this experiment offers an empirical characterization of how PET changes as the perturbation magnitude increases. For each noise level, we report the mean perturbation magnitude, the mean and standard deviation of the PET distance, and the empirical ratio $d_{\mathrm{PET}}/\delta$ averaged over all repetitions.

\begin{table}[h]
\centering
\caption{Robustness of PET under Gaussian indexed vertex perturbations. Results correspond to the mean and standard deviation over $R=50$ independent perturbations for each noise level.}
\label{tab:noise}
\begin{tabular}{ccccc}
\hline
Noise level $\mu$ & Mean $\delta$ & Mean PET distance & Std. PET distance & Mean $d_{\mathrm{PET}}/\delta$ \\
\hline
0 & 0 & 0 & 0 & -\\
0.056 & 0.20 & 28.83 & 0.30 & 144.94 \\
0.11 & 0.41 & 32.06 & 0.22 & 79.45 \\
0.17 & 0.62 & 32.94 & 0.17 & 53.53 \\
0.22 & 0.82 & 33.39 & 0.19 & 41.21 \\
0.28 & 1.03 & 33.63 & 0.18 & 32.99 \\
0.33 & 1.23 & 33.78 & 0.16 & 27.81 \\
0.39 & 1.39 & 33.78 & 0.17 & 24.51 \\
0.44 & 1.64 & 33.86 & 0.16 & 20.87 \\
0.50 & 1.86 & 33.92 & 0.19 & 18.38 \\
\hline
\end{tabular}
\end{table}

Table~\ref{tab:noise} summarizes the results. As expected, the average indexed perturbation $\delta$ increases approximately proportionally to the prescribed noise level. The PET distance increases rapidly for small perturbations and then approaches a plateau around 34, indicating that once the principal geometric features of the shape have been altered, additional perturbations produce comparatively smaller changes in the transform. Moreover, the standard deviation of the PET distance remains below $0.31$ for all noise levels, demonstrating that the measurements are highly consistent across independent perturbations. Finally, the empirical ratio $d_{\mathrm{PET}}/\delta$ decreases monotonically as the perturbation magnitude grows, suggesting that the variation of PET is sublinear with respect to the maximum indexed displacement. Overall, these observations are consistent with the expected robustness of PET under indexed vertex perturbations while highlighting its sensitivity to changes in the underlying geometry.

\begin{figure}[ht!]
     \centering
     \includegraphics[width=0.6\textwidth
     ]{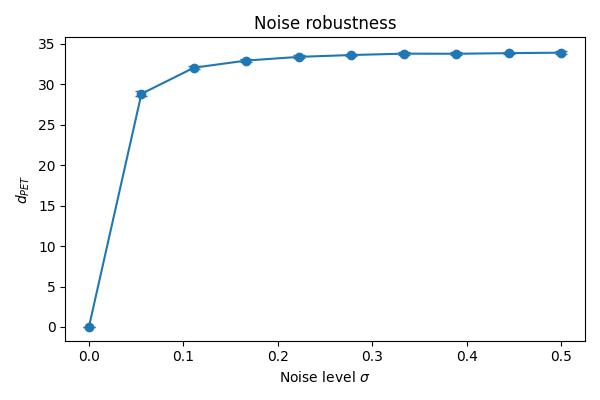}
    \caption{PET distance between the original ellipse and perturbed versions as a function of the perturbation level. The
experiment assesses how PET varies as the indexed perturbation magnitude increases} \label{fig:noiseR}
\end{figure}

\subsection{Directional Sampling Analysis}

The finite PET representation depends on the number of sampled directions. To evaluate this dependence, we compute a high-resolution reference PET using $N_{\mathrm{ref}}=256$ uniformly distributed directions and compare it with PET representations obtained from coarser directional samplings.

Let $V_{N_{\mathrm{ref}}}$ denote the reference direction set and $V_N$ a coarser set. Since the supremum over $S^1$ cannot be evaluated numerically, we approximate it using the reference discretization and define

\[
e_N =
\sup_{v\in V_{N_{\mathrm{ref}}}}
\min_{w\in V_N}
\left|
\PE(D_0(f_v))
%\mathrm{PET}_0(K(v))
-
\PE(D_0(f_w))
%\mathrm{PET}_0(K(w))
\right|.
\]

This discrete approximation mirrors the finite-direction approximation result established in Theorem~2, where the approximation error is controlled by the covering radius of the sampled direction set. For uniformly spaced directions on $S^1$, the Euclidean covering radius is
\[
\rho(V_N)
=
2\sin\left(\frac{\pi}{2N}\right).
\]

Table~\ref{tab:direction_sampling} reports the covering radius and the approximation error for increasing numbers of sampled directions. As expected, both quantities decrease monotonically as the directional sampling becomes denser. In particular, increasing the number of sampled directions from (8) to (128) reduces the approximation error from (0.1993) to (0.0153), representing a reduction of more than one order of magnitude. Moreover, the decrease in approximation error closely follows the reduction of the covering radius, providing empirical support for the approximation behaviour predicted by Theorem~2. These results indicate that relatively dense directional samplings produce PET representations that accurately approximate the high-resolution reference while requiring substantially fewer evaluated directions.

Figure~\ref{fig:conv} illustrates this convergence behaviour. The approximation error decreases steadily as the number of sampled directions increases, confirming that the finite PET representation converges toward the reference transform as the directional discretization is refined.

\begin{table}[h]
\centering
\caption{Directional sampling analysis. Approximation error is measured with respect to a high-resolution PET reference computed with \(N_{\mathrm{ref}}=256\) directions.}
\begin{tabular}{ccc}
\hline
Number of directions \(N\) & Covering radius \(\rho(V_N)\) & Approximation error \(e_N\) \\
\hline
8   & 0.39 & 0.19 \\
16  & 0.19 & 0.11 \\
32  & 0.09& 0.05 \\
64  & 0.05 & 0.03 \\
128 & 0.02 & 0.01 \\
\hline
\end{tabular}
\label{tab:direction_sampling}
\end{table}

\begin{figure}[ht!]
     \centering
     \includegraphics[width=0.6\textwidth
     ]{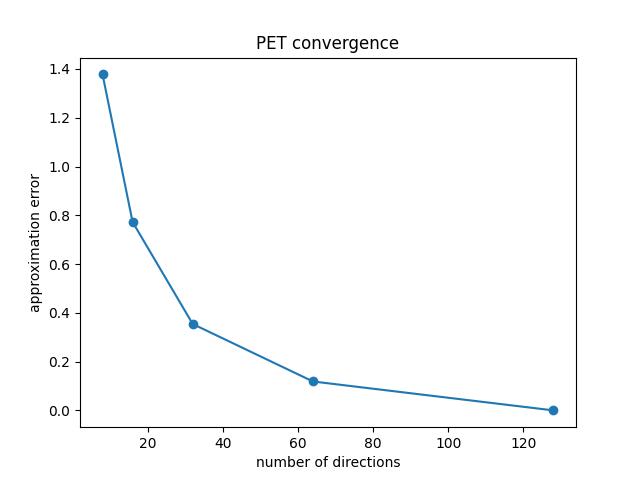}
    \caption{PET convergence} \label{fig:conv}
\end{figure}

\subsection{Classification on Real Biomedical Signals}

To further evaluate the discriminative power of the Persistent Entropy Transform (PET), we consider two electrocardiogram benchmarks: the MIT-BIH Arrhythmia dataset \cite{mitbih} and the TwoLeadECG dataset from the UCR Time Series Archive \cite{UCRArchive2018}. These datasets allow us to assess PET on real time-series benchmarks.

The MIT-BIH Arrhythmia Database is one of the most widely used benchmarks for heartbeat classification. It contains annotated ECG recordings collected from 48 subjects sampled at 360 Hz, where each heartbeat is labeled according to its cardiac rhythm. Following a binary classification setting, each heartbeat is represented by a fixed-length segment containing 400 signal samples. The original dataset exhibits a pronounced class imbalance, which may bias standard classifiers towards the majority class. Therefore, before computing any topological descriptors, the dataset is balanced by randomly undersampling the majority class, resulting in a balanced dataset containing 9\,344 heartbeats (4\,672 samples per class).

According to the UCR benchmark description, TwoLeadECG is a binary univariate ECG dataset consisting of time series of length \(82\), with an official training set of 23 instances and an official test set of 1139 instances. The recordings originate from the MIT-BIH Long-Term ECG Database (record \texttt{ltdb/15814}), and the classification task is to distinguish between two cardiac signal classes.

For both datasets, each ECG signal is represented as a planar curve by mapping the temporal index to the horizontal coordinate and the signal amplitude to the vertical coordinate. Time is normalized to the interval \([0,1]\), while amplitudes are independently normalized for each signal to remove scale differences. PET is then computed by evaluating the persistent entropy associated with lower-star filtrations induced by multiple directions on the embedded curve. Unless otherwise stated, PET is computed using \(64\) uniformly distributed directions and 0-dimensional persistence.

For the TwoLeadECG dataset, we follow the official train/test partition provided by the UCR archive. For the MIT-BIH Arrhythmia dataset, performance is estimated using repeated stratified cross-validation. The objective of these experiments is not to obtain state-of-the-art ECG classification results, but rather to evaluate PET as a compact topological descriptor. In particular, we compare PET against both the original signal representation and the classical (non-directional) persistent entropy, thereby assessing the benefit of incorporating directional topological information.

The following feature representations are evaluated:

\begin{itemize}
    \item \textbf{Raw Data}: the original ECG samples are directly used as input features.
    \item \textbf{Persistent Entropy (PE)}: a single scalar feature corresponding to the classical persistent entropy computed from the signal.
    \item \textbf{Persistent Entropy Transform (PET)}: the proposed directional descriptor, represented by a 64-dimensional feature vector.
\end{itemize}

Each representation is evaluated using three standard classifiers: a linear Support Vector Machine (SVM), Random Forest (RF), and Extreme Gradient Boosting (XGBoost). Performance is reported in terms of Accuracy, F1-score and ROC-AUC together with the dimensionality of the feature representation, since compactness is one of the main design objectives of PET.

\begin{rem}
    The objective of these experiments is not to obtain state-of-the-art ECG classification results, but rather to evaluate PET as a compact topological descriptor. In particular, we compare PET against both the original signal representation and the classical (non-directional) persistent entropy, thereby assessing the benefit of incorporating directional topological information. We do not include comparisons with the Persistent Homology Transform (PHT) or the Euler Characteristic Transform (ECT), as neither provides a directly comparable baseline in this setting. The PHT yields a collection of persistence diagrams indexed by direction rather than a fixed-dimensional feature vector, requiring an additional vectorization step (e.g., persistence images or persistence landscapes), whose choice would significantly influence the final performance. Likewise, the ECT is not naturally formulated for one-dimensional ECG signals represented through lower-star filtrations, making it unsuitable for the framework considered in this work. Our aim is therefore to isolate the contribution of the proposed PET representation as a compact vector embedding derived directly from directional persistent homology.
\end{rem}

Table~\ref{tab:mitbihresults} shows that the raw ECG representation achieves the highest overall performance, with XGBoost obtaining the best results across all evaluated metrics, reaching an accuracy of 98.62\% and an AUC of 99.90\%. However, this performance is achieved using the full 400-dimensional signal, which constitutes the largest feature representation among the evaluated methods.

As in the previous experiments, the classical persistent entropy  representation is unable to capture sufficient discriminative information. By reducing each signal to a single scalar feature,
PE
produces performance close to random guessing for all classifiers, with both accuracy and AUC remaining around 50\%. This confirms that such an aggressive dimensionality reduction removes most of the discriminative topological information present in the ECG signals.

In contrast, the proposed Persistent Entropy Transform (PET) provides a much more informative topological representation. Using only 64 features, PET achieves excellent classification performance across all three classifiers, with Random Forest obtaining the best results (96.28\% accuracy and 99.36\% AUC). Although PET does not outperform the raw 400-dimensional representation, it reduces the feature dimensionality by approximately 84\% while maintaining highly competitive performance. These results indicate that preserving the directional information of persistent entropy captures substantially richer topological characteristics than the classical PE descriptor, yielding an effective and compact representation for ECG classification.

\begin{table}[h]
\centering
\caption{Classification performance on the MIT-BIH dataset using 5-fold stratified cross-validation. Accuracy, F1-score and AUC are reported as mean $\pm$ standard deviation.}
\begin{tabular}{lcccc}
\hline
Method & Dim. & Accuracy & F1-Score & AUC \\
\hline
Raw + SVM & 400 & $0.9115 \pm 0.0052$ & $0.9117 \pm 0.0051$ & $0.9621 \pm 0.0035$ \\
Raw + RF & 400 & $0.9850 \pm 0.0025$ & $0.9849 \pm 0.0026$ & $0.9988 \pm 0.0004$ \\
Raw + XGBoost & 400 & $\mathbf{0.9862 \pm 0.0023}$ & $\mathbf{0.9862 \pm 0.0024}$ & $\mathbf{0.9990 \pm 0.0004}$ \\
\hline
PE + SVM & 1 & $0.4998 \pm 0.0001$ & $0.3999 \pm 0.3265$ & $0.5000 \pm 0.0000$ \\
PE + RF & 1 & $0.5000 \pm 0.0002$ & $0.0000 \pm 0.0000$ & $0.5000 \pm 0.0000$ \\
PE + XGBoost & 1 & $0.5000 \pm 0.0002$ & $0.0000 \pm 0.0000$ & $0.5000 \pm 0.0000$ \\
\hline
PET + SVM & 64 & $0.9516 \pm 0.0045$ & $0.9514 \pm 0.0045$ & $0.9837 \pm 0.0024$ \\
PET + RF & 64 & $\mathbf{0.9628 \pm 0.0040}$ & $\mathbf{0.9630 \pm 0.0040}$ & $\mathbf{0.9936 \pm 0.0012}$ \\
PET + XGBoost & 64 & $0.9625 \pm 0.0042$ & $0.9625 \pm 0.0042$ & $0.9933 \pm 0.0012$ \\
\hline
\end{tabular}
\label{tab:mitbihresults}
\end{table}

Table~\ref{tab:twoleadecgresults} highlights several interesting observations. As expected, the original ECG signal provides the highest classification performance, with the linear SVM achieving an accuracy of 94.21\%. Since the original signal contains only 82 samples, it already constitutes a relatively compact representation for this benchmark.

In contrast, the classical persistent entropy, represented by a single scalar feature, fails to discriminate between the two classes, yielding performance close to random guessing for all three classifiers. This result indicates that reducing the persistence information to a single entropy value removes most of the discriminative topological information contained in the signal.

The proposed PET substantially improves over the classical PE baseline for every classifier. Although PET does not consistently surpass the raw signal representation, it achieves competitive performance while using only 64 features and relying exclusively on topological information extracted from the embedded curve. In particular, Random Forest reaches an accuracy of 93.15\% and an AUC of 97.97\%, demonstrating that the directional entropy representation preserves much more discriminative information than the classical non-directional persistent entropy.

\begin{table}[h]
\centering
\caption{Classification performance on the TwoLeadECG dataset using the official UCR train/test split.  The 'Accuracy', 'F1-Score' and 'AUC' columns display mean and standard deviation values for the specified variables on the test set.}
\begin{tabular}{lccccc}
\hline
Method & Dim. & Accuracy & F1-Score & AUC \\
\hline
Raw + SVM & 82 & $\mathbf{0.9421}$  & $\mathbf{0.9420}$ & $\mathbf{0.9785}$\\
Raw + RF & 82 & 0.7191  & 0.6728 & 0.8424\\
Raw + XGBoost & 82 & 0.7814  & 0.7688 & 0.8671\\
\hline
PE + SVM & 1 & 0.4996  & 0.0000 & 0.5000\\
PE + RF & 1 & 0.4996  & 0.0000 & 0.5000\\
PE + XGBoost & 1 & 0.4996  & 0.0000 & 0.5000\\
\hline
PET + SVM & 64 & 0.8859  & 0.8902 & 0.9335\\
PET + RF & 64 & $\mathbf{0.9315}$  & $\mathbf{0.9293}$ & $\mathbf{0.9797}$\\
PET + XGBoost & 64 & 0.8982  & 0.8991 & 0.9513\\
\hline
\end{tabular}
\label{tab:twoleadecgresults}
\end{table}

These results suggest that the directional information encoded by PET is considerably more informative than the classical persistent entropy, demonstrating that the improvement is due to the directional topological representation rather than to the entropy measure itself.

\section{Conclusions}\label{sec:conclusions}
In this work, we introduced the Persistent Entropy Transform (PET), a novel topological signature combining persistent entropy with directional lower-star filtrations. The proposed framework extends persistent entropy from a global scalar descriptor to a multiscale topological signature capable of encoding geometric variability and anisotropic structural information.

From a theoretical perspective, we formalized the transform and analyzed several of its fundamental properties, including boundedness, translation invariance, rotational equivariance, and stability. These properties are inherited from the stability of persistence diagrams and persistent entropy.

The empirical validation supports the intended role of PET as a compact directional descriptor. Synthetic experiments
assess directional sensitivity, consistency with rotational equivariance, robustness under controlled perturbations, and
dependence on directional sampling density. The TwoLeadECG and MIT-BIH experiments provides a proof of concept that PET
embeddings can be used as compact vector features on a real time-series benchmark.

PET is not intended to replace fully expressive directional transforms such as PHT or ECT when injectivity, reconstruction, or complete geometric characterization is required. Instead, it provides a complementary descriptor for
settings in which compactness, interpretability, and compatibility with standard statistical learning methods are primary
constraints.

\paragraph*{\textbf{Future work}}

Several research directions arise naturally from this work. First, a deeper theoretical analysis of the injectivity and discriminative power of PET remains open. In particular, understanding the extent to which PET characterizes geometric objects is an interesting problem connected with directional topological transforms such as the Persistent Homology Transform.
Future work will also explore applications of PET in machine learning, topological signal processing, and geometric representation learning, as well as its integration with differentiable topological pipelines, as it is a natural candidate for integration into differentiable topological layers for geometric deep learning.

% PET is not known to be injective in general. A natural candidate
% for a counterexample would be a pair of shapes $\mathcal{K}_1,
% \mathcal{K}_2$ such that $PE(\mathrm{D}_k(f_v^{(1)})) =
% PE(\mathrm{D}_k(f_v^{(2)}))$ for all $v \in S^{d-1}$ despite
% $\mathrm{D}_k(f_v^{(1)}) \neq \mathrm{D}_k(f_v^{(2)})$ for some
% $v$; informally, shapes that differ in the relative ordering of
% their persistence intervals across directions but share identical
% entropy profiles. Constructing such a pair explicitly, or proving
% that none exists under geometric regularity conditions, remains
% open.

\paragraph*{\textbf{Code Availability}}
All code for the proposed methodology, as well as for generate the results presented in this manuscript, are publicly available in a Github repository \footnote{\url{https://github.com/victosdur/PET}}.

%\paragraph*{\textbf{Acknowledgements}} 
%We want to thank the reviewers, who gave us useful comments to improve the content of this paper, as well as ideas for future work. 
%This work was partially supported by REXASI-PRO H-EU project, call HORIZON-CL4-2021-HUMAN-01-01, Grant agreement ID: 101070028.

\paragraph*{\textbf{Disclosure of Interest}}
The authors have no competing interests to declare that are relevant to the content of this article.

\paragraph*{\textbf{Acknowledgement}}
This work was partially supported by project PID2025-171542NB-I00, funded by MICIU/AEI/10.13039/501100011033.
% It is now necessary to declare any competing interests or to specifically
% state that the authors have no competing interests. Please place the
% statement with a bold run-in heading in small font size beneath the
% (optional) acknowledgments\footnote{If EquinOCS, our proceedings submission
% system, is used, then the disclaimer can be provided directly in the system.},
% for example: The authors have no competing interests to declare that are
% relevant to the content of this article. Or: Author A has received research
% grants from Company W. Author B has received a speaker honorarium from
% Company X and owns stock in Company Y. Author C is a member of committee Z.

\bibliographystyle{plain}
\bibliography{biblio}
% \nocite{*}

\end{document}